\documentclass[12pt,reqno]{amsart}
\usepackage{fullpage,  amssymb, amsthm, amsmath, amsfonts, times}

\usepackage{enumitem}
\usepackage{mathrsfs}
\usepackage{mathtools}
\usepackage{centernot}
\usepackage{xfrac}
\usepackage{eufrak}
\usepackage{stmaryrd}
\usepackage{bm}
\usepackage[all,cmtip]{xy}
\usepackage[usenames,dvipsnames]{xcolor}
\usepackage[margin=1.0 in]{geometry}
\usepackage{float}
\usepackage{hyperref}
\hypersetup{colorlinks=true, citecolor=blue, linkcolor=blue, urlcolor=blue, pdfstartview=FitH, pdfauthor=Curran, pdftitle=}

\usepackage{tikz-cd}
\usetikzlibrary{positioning,arrows}
\usetikzlibrary{calc}

\usepackage{todonotes}

\usepackage[utf8]{inputenc}

\begin{document}

\parskip0pt
\parindent10pt

\newenvironment{answer}{\color{Blue}}{\color{Black}}
\newenvironment{exercise}
{\color{Blue}\begin{exr}}{\end{exr}\color{Black}}

\theoremstyle{plain} 
\newtheorem{theorem}{Theorem}[section]
\newtheorem*{theorem*}{Theorem}
\newtheorem{prop}[theorem]{Proposition}
\newtheorem{porism}[theorem]{Porism}
\newtheorem{lemma}[theorem]{Lemma}
\newtheorem{cor}[theorem]{Corollary}
\newtheorem{conj}[theorem]{Conjecture}
\newtheorem{funfact}[theorem]{Fun Fact}
\newtheorem*{claim}{Claim}
\newtheorem{question}{Question}
\newtheorem*{conv}{Convention}

\theoremstyle{remark}
\newtheorem{exr}{Exercise}
\newtheorem*{rmk}{Remark}

\theoremstyle{definition}
\newtheorem{defn}{Definition}
\newtheorem{example}{Example}

\renewcommand{\mod}[1]{{\ifmmode\text{\rm\ (mod~$#1$)}\else\discretionary{}{}{\hbox{ }}\rm(mod~$#1$)\fi}}

\newcommand{\ns}{\mathrel{\unlhd}}
\newcommand{\tr}{\text{tr}}
\newcommand{\wt}[1]{\widetilde{#1}}
\newcommand{\wh}[1]{\widehat{#1}}
\newcommand{\cbrt}[1]{\sqrt[3]{#1}}
\newcommand{\floor}[1]{\left\lfloor#1\right\rfloor}
\newcommand{\abs}[1]{\left|#1\right|}
\newcommand{\ds}{\displaystyle}
\newcommand{\nn}{\nonumber}
\newcommand{\re}{\text{Re}}
\renewcommand{\ker}{\textup{ker }}
\renewcommand{\char}{\textup{char }}
\renewcommand{\Im}{\textup{Im }}
\renewcommand{\Re}{\textup{Re }}
\newcommand{\area}{\textup{area }}
\newcommand{\isom}
    {\ds \mathop{\longrightarrow}^{\sim}}
\renewcommand{\ni}{\noindent}
\renewcommand{\bar}{\overline}
\newcommand{\morph}[1]
    {\ds \mathop{\longrightarrow}^{#1}}

\newcommand{\Gal}{\textup{Gal}}
\newcommand{\Aut}{\textup{Aut}}
\newcommand{\Crypt}{\textup{Crypt}}
\newcommand{\disc}{\textup{disc}}
\newcommand{\sgn}{\textup{sgn}}
\newcommand{\del}{\partial}

\newcommand{\mattwo}[4]{
\begin{pmatrix} #1 & #2 \\ #3 & #4 \end{pmatrix}
}

\newcommand{\vtwo}[2]{
\begin{pmatrix} #1 \\ #2 \end{pmatrix}
}
\newcommand{\vthree}[3]{
\begin{pmatrix} #1 \\ #2 \\ #3 \end{pmatrix}
}
\newcommand{\vcol}[3]{
\begin{pmatrix} #1 \\ #2 \\ \vdots \\ #3 \end{pmatrix}
}

\newcommand*\wb[3]{%
  {\fontsize{#1}{#2}\usefont{U}{webo}{xl}{n}#3}}

\newcommand\myasterismi{%
  \par\bigskip\noindent\hfill
  \wb{10}{12}{I}\hfill\null\par\bigskip
}
\newcommand\myasterismii{%
  \par\bigskip\noindent\hfill
  \wb{15}{18}{UV}\hfill\null\par\medskip
}
\newcommand\myasterismiii{%
  \par\bigskip\noindent\hfill
  \wb{15}{18}{z}\hfill\null\par\bigskip
}

\newcommand{\one}{{\rm 1\hspace*{-0.4ex} \rule{0.1ex}{1.52ex}\hspace*{0.2ex}}}

\renewcommand{\v}{\vec{v}}
\newcommand{\w}{\vec{w}}
\newcommand{\e}{\vec{e}}
\newcommand{\m}{\vec{m}}
\renewcommand{\u}{\vec{u}}
\newcommand{\vecx}{\vec{e}_1}
\newcommand{\vecy}{\vec{e}_2}
\newcommand{\vo}{\vec{v}_1}
\newcommand{\vt}{\vec{v}_2}

\renewcommand{\o}{\omega}
\renewcommand{\a}{\alpha}
\renewcommand{\b}{\beta}
\newcommand{\g}{\gamma}
\newcommand{\sig}{\sigma}
\renewcommand{\d}{\delta}
\renewcommand{\t}{\theta}
\renewcommand{\k}{\kappa}
\newcommand{\ve}{\varepsilon}
\newcommand{\op}{\text{op}}

\newcommand{\Z}{\mathbb Z}
\newcommand{\ZN}{\Z_N}
\newcommand{\Q}{\mathbb Q}
\newcommand{\N}{\mathbb N}
\newcommand{\R}{\mathbb R}
\newcommand{\C}{\mathbb C}
\newcommand{\F}{\mathbb F}
\newcommand{\T}{\mathbb T}
\renewcommand{\H}{\mathbb H}
\newcommand{\B}{\mathcal B}
\newcommand{\p}{\mathcal P}
\renewcommand{\P}{\mathbb P}
\renewcommand{\r}{\mathcal R}
\renewcommand{\c}{\mathcal C}
\newcommand{\h}{\mathcal H}
\newcommand{\f}{\mathcal F}
\newcommand{\s}{\mathcal S}
\renewcommand{\L}{\mathcal L}
\newcommand{\lam}{\lambda}
\newcommand{\E}{\mathcal E}
\newcommand{\Ex}{\mathbb E}
\newcommand{\D}{\mathbb D}
\newcommand{\oh}{\mathcal O}
\newcommand{\n}{\mathcal N}
\newcommand{\I}{\mathcal I}
\newcommand{\G}{\mathcal G}

\newcommand{\diam}{\text{ diam}}
\newcommand{\vol}{\text{vol}}
\newcommand{\Int}{\text{Int}}

\newcommand{\Span}{\text{span}}
\newcommand{\AP}{\text{AP}}

\newcommand{\MoM}{\text{MoM}}

\newcommand{\0}{{\vec 0}}

\newcommand{\ignore}[1]{}

\newcommand{\poly}[1]{\textup{Poly}_{#1}}

\newcommand*\circled[1]{\tikz[baseline=(char.base)]{
            \node[shape=circle,draw,inner sep=2pt] (char) {#1};}}

\newcommand*\squared[1]{\tikz[baseline=(char.base)]{
            \node[shape=rectangle,draw,inner sep=2pt] (char) {#1};}}

\title{Correlations of the Riemann zeta function} 

\author{Michael J. Curran }
\email{Michael.Curran@maths.ox.ac.uk}
\address{Mathematical Institute, University of Oxford, Oxford, OX2 6GG, United Kingdom.}

\maketitle

\begin{abstract}
Assuming the Riemann hypothesis, we investigate the shifted moments of the zeta function
\[
M_{\bm{\a},\bm{\b}}(T) = \int_T^{2T}  \prod_{k = 1}^m |\zeta(\tfrac{1}{2} + i (t + \a_k))|^{2 \b_k} dt
\]
introduced by Chandee \cite{Chandee}, where $\bm{\a} = \bm{\a}(T) =  (\a_1, \ldots, \a_m)$ and $\bm{\b} = (\b_1 \ldots , \b_m)$ satisfy $|\a_k| \leq T/2$ and $\b_k\geq 0$.
We shall prove that
\[
M_{\bm{\a},\bm{\b}}(T)  \ll_{\bm{\b}} T (\log T)^{\b_1^2 + \cdots + \b_m^2} \prod_{1\leq j < k \leq m} |\zeta(1 + i(\a_j - \a_k) + 1/ \log T )|^{2\b_j \b_k}.
\]
This improves upon the previous best known bounds due to Chandee \cite{Chandee} and Ng, Shen, and Wong \cite{NSW}, particularly when the differences $|\a_j - \a_k|$ are unbounded as $T\rightarrow \infty$.
The key insight is to combine work of Heap, Radziwiłł, and Soundararajan \cite{HRS} and work of the author \cite{CurranMoM} with the work of Harper \cite{Harper} on the moments of the zeta function.

\end{abstract}

\section{Introduction}\label{sec:Intro}

This paper is concerned with the shifted moments 

\begin{equation}\label{eqn:shiftedMoments}
M_{\bm{\a},\bm{\b}}(T) = \int_T^{2T}  \prod_{k = 1}^m |\zeta(\tfrac{1}{2} + i (t + \a_k))|^{2 \b_k} dt, 
\end{equation}
where $\bm{\a} = \bm{\a}(T) =  (\a_1, \ldots, \a_m)$ and $\bm{\b} = (\b_1 \ldots , \b_m)$ satisfy $|\a_k| \leq T/2$ and $\b_k\geq 0$.
These were first studied in general by Chandee \cite{Chandee}, who gave lower bounds assuming the $\b_k$ are integers, $\a_k= O(\log \log T)$, and $|\a_j - \a_k| = O(1)$. Chandee also gave upper bounds assuming the Riemann hypothesis when $|\a_j - \a_k| = O(1)$ and $\a_k= O(\log T)$ which are sharp up to a $(\log T)^\ve$ loss.
Subsequently Ng, Shen, and Wong \cite{NSW} removed the $(\log T)^\ve$ loss in the special case where $\bm{\b} = (\b,\b)$  by using the work of Harper \cite{Harper} on the moments of the zeta, and they also gave bounds in the larger regime $|\a_1 + \a_2| \leq T^{0.6}$.
More precisely, in this range they proved 
\[
M_{(\a_1, \a_2),(\b,\b)}(T)  \ll T (\log T)^{2\b^2} F(\a_1,\a_2, T)^{2\b^2}
\]
where 
\[
F(\a_1,\a_2, T) = 
\begin{cases}
\min\left(|\a_1 - \a_2|^{-1}, \log T\right) & |\a_1 - \a_2| \leq 1/100\\
\log(2 + |\a_1 - \a_2|) & |\a_1 - \a_2| > 1/100 
\end{cases}.
\]

Some special cases of the shifted moments $M_{\bm{\a},\bm{\b}}(T)$ and related objects have been studied unconditionally. 
For example the integral
\[
\int_T^{2T} \zeta (\tfrac{1}{2}+ i (t + \a_1))\zeta (\tfrac{1}{2}- i (t + \a_2)) dt
\]
akin to $M_{\bm{\a},\bm{\b}}(T)$  with $\bm{\b} = (\tfrac{1}{2},\tfrac{1}{2})$ is fairly well understood.
Here the current state of the art comes from Atkinson's formula for the mean square of zeta \cite{Atkinson} and Bettin's work on the second moment of zeta with  shifts of size $T^{2 - \ve}$ \cite{BettinUS}. 
The current state of the art for $M_{\bm{\a},\bm{\b}}(T)$ with $\bm{\b} = (1,1)$ is due to Motohashi's  explicit formula for the fourth moment of zeta \cite{Motohashi4th, MotohashiSpec} and Kovaleva's work on the fourth moment of zeta with shifts of size up to $T^{3/2 - \ve}$ \cite{Kovaleva}.
Finally in the case where $\bm{\b} = (\b,\b)$, sharp upper bounds for $\b \leq 1$ and lower bounds for all $\b \geq 0$ with shifts of size up to $T^{1/2-\ve}$ were obtained by the author \cite{CurranMoM}.
The goal of this paper is, assuming the Riemann hypothesis, to extend the work of Ng, Shen, and Wong \cite{NSW} to arbitrary $\bm{\a}$ and $\bm{\b}$ and to give stronger bounds in the regime where the differences $|\a_j - \a_k|$ are unbounded.

\begin{theorem}\label{thm:main}
Assume the Riemann hypothesis. If $\b_k \geq 0$ and $|\a_k| \leq T/2$ for $k = 1, \ldots , m,$ then
\[
M_{\bm{\a},\bm{\b}}(T) \ll_{\bm{\b}} T (\log T)^{\b_1^2 + \cdots + \b_m^2} \prod_{1\leq j < k \leq m} |\zeta(1 + i(\a_j - \a_k) + 1/ \log T )|^{2\b_j \b_k}.
\]
\end{theorem}

\begin{rmk}
Throughout this paper we will assume $T$ is sufficiently large in terms of $\bm{\b}$.  
\end{rmk}

Our bound is the same order of magnitude predicted by the famous recipe of Conrey, Farmer, Keating, Rubinstein, and Snaith \cite{CFKRS}.
We obtain lower bounds of the same order in a subsequent paper \cite{CurranCLB}, so the bound is sharp.
At heart, Theorem \ref{thm:main} is a statement about how $\zeta(\tfrac{1}{2} + i t)$ and $\zeta(\tfrac{1}{2} + i (t + \a))$ are correlated for $t \in  [T,2T]$ and $|\a|\leq T/2$.
More precisely, it predicts that $\zeta(\tfrac{1}{2} + i t)$ and $\zeta(\tfrac{1}{2} + i (t + \a))$ are perfectly correlated on average when $|\a| \leq 1 / \log T$, and decorrelate like $|\zeta(1 + i \a)|$ for $|\a| > 1 / \log T$. When $\a \leq 1$, the Laurent expansion for zeta shows that we obtain the same correlations predicted from random matrix theory. 
For larger $\a$, the correlation is of order 1 on average, which can be seen by calculating the moments of zeta to the right of the 1-line.
There are, however, long range correlations coming from the primes; more precisely, from the extreme values of zeta on the one line.
This is not so surprising, for the Keating Snaith philosophy only predicts that random matrix theory is a good model for $\zeta(\tfrac{1}{2}+ i t)$ in short intervals.

The starting point of the proof is to use the method of Soundararajan \cite{SRHMoments} and Harper \cite{Harper} to bound $\log |\zeta(\tfrac{1}{2} + i t)|$ by a short Dirichlet polynomial.
Instead of following the argument of Harper, however, we treat the exponential of this short Dirichlet polynomial in a manner similar to the approach taken in the work of Heap, Radziwiłł, and Soundararajan \cite{HRS}.
Using this method, the integrals that arise can be evaluated by simply using the mean value theorem for Dirichlet polynomials. 
These mean values are much easier to evaluate uniformly in the shifts $\a_k$ than the integrals of products of shifted cosines that appear when using Harper's method \cite{Harper} as Ng, Shen, and Wong  do in \cite{NSW}.
This difference allows us to obtain upper bounds for general shifts $\bm{\a}$ and exponents $\bm{\b}$. 
The final ingredient is a more precise estimate of the following sum
\[
\sum_{p\leq X} \frac{\cos(\d \log p)}{p} 
\]
coming from the theory of pretentious multiplicative functions, see Lemma \ref{lem:cosPrimeSum}.
This idea appeared in the author's previous work on studying the second moment of moments of zeta in short intervals \cite{CurranMoM}.
This more precise estimate is what allows us to improve the bound of Ng, Shen, and Wong \cite{NSW} in the regime where $|\a_j - \a_k|$ is unbounded.

\section*{Acknowledgements}
\noindent
The author would like to thank Paul Bourgade, Hung Bui, and Maksym Radziwiłł for bringing the author's attention to the works of Chandee \cite{Chandee} and Ng, Shen, and Wong \cite{NSW}. The author would also like to thank his advisor Jon Keating, Valeriya Kovaleva, and Maksym Radziwiłł for useful comments.

\section{Preliminary tools and notation}\label{sec:Notation}

We will start by using the following lemma, which is due to Soundararajan \cite{SRHMoments} and Harper \cite{Harper}.

\begin{lemma}\label{lem:logZetaUpperBound}
Assume the Riemann hypothesis, let $t\in [T,2T]$, and $|\a| \leq T/2$. Then for $2\leq X \leq T^2$ 
\begin{align*}
\log|\zeta(\tfrac{1}{2} &+ i (t + \a))|\leq \Re \sum_{p \leq X} \frac{1}{p^{1/2 + 1/\log X+ i(t + \a)}} \frac{\log X/p}{\log X}  
\\ + &\sum_{p\leq \min(\sqrt{X}, \log T)} \frac{1}{2 p^{1 + 2 i (t +\a)}} + \frac{\log T}{\log X} + O(1).
\end{align*}
\end{lemma}
\noindent
Throughout it will be useful to break the set of primes into certain intervals.
Set
\[
\b_\ast := \sum_{k\leq m} \max(1, \b_k).
\]
Throughout this paper we will use $\log_j$ to denote the $j$-fold iterated logarithm.
We choose a sequence of parameters $T_j = T^{c_j}$, where
\[
c_0 = 0 \text{ and } c_j = \frac{e^{j}}{(\log_2 T)^2}
\]
for $j > 0$.
We will choose $L$ to be the largest integer such that $T_L \leq T^{e^{-1000\b_\ast}}$.
Let 
\[
\p_{1,X}(s) = \sum_{p \leq T_1} \frac{1}{p^{s+1/\log X}} \frac{\log X/p}{\log X} + \sum_{p\leq \log T} \frac{1}{2p^{2s}},
\]
and given any $2\leq j \leq L$ define
\[
\p_{j,X}(s) = \sum_{p\in (T_{j-1},T_j]} \frac{1}{p^{s+1/\log X}} \frac{\log X/p}{\log X}. 
\]
If $\p_{j,X}(s)$ is not too large, then we will be able to efficiently approximate $\exp(\b \p_{j,X} (s))$ with its Taylor series.
Indeed, if we  choose cutoff parameters $K_j = c_j^{-3/4}$ for $j \geq 1$ and set
\[\label{eq:TaylorExp}
\n_{j,X}(s;\b) := \sum_{m \leq  100\b_\ast^2 K_j} \frac{\b^m \p_{j,X}(s)^m}{m!}
\]
then we have the following analog of lemma 1 of \cite{HRS}: 

\begin{lemma}\label{lem:expTaylorSeries}
If $\b \leq \b_\ast$ and $|\p_{j,X}(s)| \leq K_j$ for some $1\leq j \leq L$, then
\[
\exp(2\b \Re \p_{j,X}(s)) \leq (1+ e^{-100\b_\ast^2 K_j})^{-1} |\n_{j,X}(s;\b)|^2.
\]
\end{lemma}
\begin{proof}
Since $|\p_{j,X}(s)| \leq 2 K_j$, Taylor expansion gives
\[
|\exp(\b \p_{j,X}(s))| -  e^{-100\b_\ast^2 K_j} \leq |\n_{j,X}(s;\b)|.
\]
By assumption $\exp(-2 K_j \b_\ast) \leq |\exp(\b \p_{j,X}(s))| \leq \exp(2 K_j \b_\ast)$, so the claim readily follows.
\end{proof}

We will first bound the shifted moment of zeta when all of the shifts $t + \a_k$ lie in the ``good" set
\begin{equation}
\G := \left\{t\in[T/2,5T/2]: |P_{j,T_L} (\tfrac{1}{2} + i t)| \leq K_j \text{ for all } 1\leq j\leq L\right\}.
\end{equation}
In this case we may use Lemma \ref{lem:logZetaUpperBound} with $X = T_L$ in tandem with Lemma \ref{lem:expTaylorSeries} to reduce the problem to computing the mean value of certain Dirichlet polynomial.
We will accomplish this with the following mean value theorem of Montgomery and Vaughan (see for example theorem 9.1 of \cite{IK}).
\begin{lemma}\label{lem:MVDP}
Given any complex numbers $a_n$
\[
\int_{T}^{2T} \left|\sum_{n\leq N} \frac{a_n}{n^{i t}}\right|^2 dt = (T + O(N)) \sum_{n\leq N} |a_n|^2.
\]
\end{lemma}
\noindent
We will also make use of the property that  Dirichlet polynomials supported on distinct sets of primes are approximately independent in the mean square sense. 
The precise formulation we will use is the following splitting lemma which appears in equation (16) of \cite{HS}.
\begin{lemma}\label{lem:Splitting}
Suppose for $1\leq j \leq J$ we have $j$ disjoint intervals $I_j$ and Dirichlet polynomials $A_j(s)= \sum_{n} a_j(n)n^{-s}$ such that $a_j(n)$ vanishes unless $n$ is composed of primes in $I_j$.
Then if $\prod_{j\leq J} A_j(s)$ is a Dirichlet polynomial of length $N$
\begin{align*}
\int_T^{2T} \prod_{j\leq J} |A_j(\tfrac{1}{2} + i t)|^2 dt= (T + O(N))\prod_{j\leq J}\left(\frac{1}{T}\int_T^{2T}|A_j(\tfrac{1}{2} + i t)|^2 dt\right)
\end{align*}
\end{lemma}
\noindent
The following variant due to Soundararajan \cite[lemma 3]{SRHMoments} will also be useful for handling moments of Dirichlet polynomials supported on primes.
\begin{lemma}\label{lem:MVDPPrimes}
Let $r$ be a natural number and suppose $N^r \leq T/\log T$. Then given any complex numbers $a_p$
\[
\int_{T}^{2T} \left|\sum_{p\leq N} \frac{a_p}{p^{i t}}\right|^{2r} dt \ll T r! \left(\sum_{p\leq N} |a_p|^2\right)^r.
\]
\end{lemma}
\noindent
During the main mean value calculation, we will need to bound a certain product over primes. This product will be controlled with the following lemma, which is a special case of lemma 3.2 of \cite{Koukoulopoulos}.

\begin{lemma}\label{lem:cosPrimeSum}
Given $\d \in \R$ and $X \geq 2$  
\begin{align*}
\sum_{p\leq X} &\frac{\cos(\d \log p)}{p}  = \log|\zeta(1 + 1/\log X + i\d)| + O(1).
\end{align*}
\end{lemma}

To handle the shifted moment of zeta when some of the shifts $t + \a_k$ lie in the ``bad" set $[T/2,5T/2]\setminus \G$, we take advantage of the incremental structure present.
For each $1 \leq j \leq L$, define
\begin{align*}
\B_j := \{t\in [T/2&,5T/2]: |\p_{r,T_s} (\tfrac{1}{2} + i t)| \leq K_j \text{ for all } 1\leq r < j\text{ and } r\leq s \leq  L \\
&\text{ but } |\p_{j,T_s} (\tfrac{1}{2} + i t)|  > K_{j} \text{ for some } j \leq s \leq L \}.
\end{align*}
Notice that
\[
[T/2, 5T/2] \setminus \G = \bigsqcup_{j \leq L} \B_j.
\]
On the bad sets $\B_j$ the series expansion $\n_{j,T_s}$ of $\exp(P_{j,T_s})$ is a poor approximation, so we are forced to estimate $\log \zeta$ using only the primes up to $T_{j-1}$.
While the resulting Dirichlet polynomial is too short to obtain sharp bounds, we can overcome this loss by multiplying by a suitably large even power of $|P_{j,T_s}| / K_{j}$, which is larger than 1 on $\B_j$. 
If we then extend the range of integration to all of $[T,2T]$,  we can still win as the event $|\p_{j,T_s} (\tfrac{1}{2} + i t)|  > K_{j}$ is quite rare.
For example, we will make use of the following bound.
\begin{lemma}\label{lem:measB1}
If $T_1^r \leq T/\log T$ then
\[
\int_{T/2}^{5T/2} |\p_{1,X}(\tfrac{1}{2} + i t)|^{2 r} dt \ll 2^{2r} r! T (\log_2 T)^r.
\]
Therefore
\[
\text{meas}(\B_1) \ll T e^{-K_1^2/4\log_2 T} \ll_A T(\log T)^{-A}.
\]
\end{lemma}
\begin{proof}

Write $\p_{1,X} = \p_{1,X}^{(1)} + \p_{1,X}^{(2)}$, where $\p_{1,X}^{(1)}$ is the sum of primes up to $T_1$ and $\p_{1,X}^{(2)}$ is the sum of squares of primes up to $\log T$. 
Then
\[
\int_{T/2}^{5T/2} |\p_{1,X}(\tfrac{1}{2} + i t)|^{2 r} dt \leq 2^{2r} \int_{T/2}^{5T/2} |\p_{1,X}^{(1)}(\tfrac{1}{2} + i t)|^{2 r} dt + 2^{2r}\int_{T/2}^{5T/2} |\p_{1,X}^{(2)}(\tfrac{1}{2} + i t)|^{2 r} dt.
\]
By Lemma \ref{lem:MVDPPrimes}, this is at most
\[
\ll 2^{2r} r! T (\log_2 T_1 + O(1))^r + 2^{2r} r! T (\zeta(2)/4)^r \ll  2^{2r} r! T (\log_2 T)^r.
\]
To deduce the second bound we note that
\[
\text{meas}(\B_1) \ll \max_{s\leq L} \frac{1}{K_1^{2 r}} \int_{T/2}^{5T/2} |\p_{1,T_s}(\tfrac{1}{2} + i t)|^{2 r} dt.
\]
We may now conclude by taking $r = \lceil K_1^2/4 \log_2 T \rceil$ and using Stirling's approximation.
\end{proof}

The proof of Theorem \ref{thm:main} is based on the following partition of $[T,2T]$: Given a subset $A$ of $[m] := \{1,\ldots, m\}$ define 
\[
\G_A :=  \left\{t\in [T,2T]: t + \a_k \in \G \text{ if and only if } k \in A\right\}.
\]
Then we can decompose $[T,2T]$ into the disjoint union
\begin{equation}\label{eqn:GoodBadPartition}
[T,2T] = \bigsqcup_{A \subseteq [m]} \G_A.
\end{equation}
In section \ref{sec:GoodShifts}, we will handle the integral over the set $\G_{[m]}$ where all of the shifts $t + \a_k$ are good.
In  section \ref{sec:BadShifts}, we will handle the cases where some of the shifts $t + \a_k$ are bad.
We will have to further partition the sets $\G_A$ with $A\subsetneq [m]$ according to which of the sets $\B_j$ the bad shifts $t + \a_k$ lie in.

\section{Moments over good shifts}\label{sec:GoodShifts}

By Lemma \ref{lem:logZetaUpperBound} with $X = T_L$ we find
\begin{align*}
\int_{\G_{[m]}} \prod_{k = 1}^m |\zeta(\tfrac{1}{2} + i (t + \a_k))|^{2 \b_k} dt 
\ll_{\bm{\b}} \int_{\G_{[m]}} \prod_{k = 1}^m \exp\left(2\b_k \Re \sum_{j = 1}^L  \p_{j,T_L} (\tfrac{1}{2} + i(t + \a_k)) \right) dt
\end{align*}
By definition of $\G_{[m]}$ the hypotheses of Lemma \ref{lem:expTaylorSeries} are satisfied for all $j\leq L$, so the integral over $\G_{[m]}$ can be bounded by 
\begin{align}\label{eqn:goodShiftDPMV}
\ll_{\bm{\b}} \int_{\G_{[m]}} &\prod_{k = 1}^m  \prod_{j = 1}^L (1+ e^{-50\b_\ast^2 K_j})^{-1}  |\n_{j,T_L}(\tfrac{1}{2} + i (t + \a_k) ; \b_k)|^2 dt \nonumber\\
&\ll_{\bm{\b}} \int_T^{2T} \prod_{k = 1}^m  \prod_{j = 1}^L   |\n_{j,T_L}(\tfrac{1}{2}  + i (t + \a_k) ; \b_k)|^2 dt.
\end{align}

We are now in a setting where me may use the mean value theorem for Dirichlet polynomials.
First note that $\prod_{k = 1}^m \n_{j, T_L}(s+ i\a_k; \b_k)$ has length at most $T_j^{200m \b_\ast^2 K_j}$ for $j = 1$ and $T_j^{100m\b_\ast^2  K_j}$ for $2\leq j \leq L$. 
Therefore the integrand $\prod_{j\leq L}\prod_{k = 1}^m \n_{j, T_L}(s+ i\a_k; \b_k)$ has length at most $T_1^{200m\b_\ast^2 K_1} T_2^{100m\b_\ast^2 K_2} \cdots T_L^{100m\b_\ast^2 K_L} \leq T^{1/2}$ as $T_L \leq T^{e^{-1000\b_\ast}}$, so it is a short Dirichlet polynomial.
By Lemma \ref{lem:Splitting}, we are left with the task of computing
\[
\int_{T}^{2T} \prod_{k = 1}^m |\n_{j, T_L}(\tfrac{1}{2}  + i (t + \a_k); \b_k)|^2 dt
\]
for each $1\leq j \leq L$.
To do this, we must analyze the coefficients of the $\n_{j,X}$. 
Denote $a_X(p):= \log(X/p)  p^{-1/\log X}/\log X$ and define multiplicative functions $g_X$ and $h_X$ satisfying
\[
g_X(p^r;\b) :=  \frac{\b^r a_X(p)^{r}}{r!}, 
\]
and
\[
h_X(p^r;\b) :=  g_X(p^r; \b) + 1_{p\leq \log T} \sum_{t = 1}^{r/2} \frac{\b^{r-t} a_X(p)^{r-t}}{2^t t! (r-2t)!} .
\]
Next define $c_{1}(n)$ to be 1 if $n$ can be written as $n = n_1 \cdots n_r$ where $r \leq 100\b_\ast^2 K_1$ and each $n_i$ is either a prime $\leq T_1$ or a prime square $\leq \log T$. 
Finally for $2\leq j \leq L$ set $c_{j}(n)$ to be 1 if $n$ is the product of at most $100\b_\ast^2 K_j$ not necessarily distinct primes in $(T_{j-1},T_j]$.

\begin{prop}\label{prop:Njcoeffs}
For $2\leq j \leq L$
\[
\n_{j,X}(s;\b) = \sum_{p\mid n \Rightarrow p \in (T_{j-1},T_j] } \frac{g_X(n;\b)c_{j}(n)}{n^s}.
\]
If
\[
\n_{1,X}(s;\b) = \sum_{p\mid n \Rightarrow p \in (T_{j-1},T_j]} \frac{f_X(n;\b)}{n^s}
\]
then $f_X(n;\b) \leq h_X(n;\b) c_{1}(n)$ and $f_X(p;\b) = g_X(p;\b)$.
\end{prop}
\begin{proof}
When $j \geq 1$ write $p_1,\ldots , p_a$ for the primes in $(T_{j-1},T_j]$. 
First assume $j\geq 2$.
By applying the multinomial theorem to the definition of $\n_{j,X}(s;\b)$ we find it equals
\[
\sum_{m\leq 100\b_\ast^2 K_j} \frac{\b^m}{m!} \left(\sum_{p\in (T_{j-1}, T_j]} \frac{a_X(p)}{p^{s}}\right)^m = \sum_{m\leq K} \frac{\b^m}{m!} \sum_{\substack{u_1 + \ldots + u_a = m \\ u_r \geq 0}} \binom{m}{u_1, \ldots , u_a} \prod_{r = 1}^a \frac{a_X(p)^{u_r}} {p_r^{u_r s}}.
\]
Therefore if $n = p_1^{u_1} \cdots p_{a}^{u_{a}}$ with $u_1 + \cdots + u_a = m$, the coefficient of $n^{-s}$ in $\n_{j,X}(s)$ is
\[
c_{j}(n) \frac{\b^m}{m!} \binom{m}{u_1, \ldots, u_a} \prod_{r = 1}^a a_X(p)^{u_r}  =  g_X(n;\b)c_{j}(n).
\]
Next we will handle the case of $j = 1$. Now we will also denote the primes up to $\log T$ by $p_1, \ldots p_b$ with $b < a$. 
The multinomial theorem tells us that $\n_{1,X}$ equals 
\begin{align*}
&\sum_{m\leq 100\b_\ast^2 K_j} \frac{\b^m}{m!} \left(\sum_{p \leq T_1} \frac{a_X(p)}{p^{s}} + \sum_{p\leq \log T} \frac{1}{2p^{2s}}\right)^m \\
= \sum_{m\leq K} \frac{\b^m}{m!} &\sum_{\substack{u_1 + \ldots + u_a + v_1 + \ldots + v_b = m \\ u_r, v_r\geq 0}} \binom{m}{u_1, \ldots , u_a, v_1 , \ldots, v_b} \prod_{r = 1}^a \frac{a_X(p)^{u_r}} {p_r^{u_r s}} \prod_{r = 1}^b \frac{1}{2^{v_r} p_r ^{2 v_r s}}.
\end{align*}
The claim now follows by considering the possible ways to write $n = p_1^{u_1} \cdots p_{a}^{u_{a}}$ as a product of the $p_r$ with $r\leq a$ or of $p_r^2$ with $r\leq b.$ 

\end{proof}

We may write
\[
\prod_{k = 1}^m \n_{j, X}(s + i\a_k; \b_k) = \sum_{n\geq 1} \frac{b_{j, X, \bm{\a}, \bm{\b}} (n)}{n^s},
\]
where $b_{1,X,\bm{\a},\bm{\b}}(n)$ is the $m$-fold Dirichlet convolution of $f_X(n;\b_k) n^{-i\a_k}$ and $b_{j,X,\bm{\a},\bm{\b}}(n)$ is the $m$-fold convolution of $g_X(n;\b_k) c_{j}(n) n^{-i\a_k}$ for $2 \leq j \leq L$.
For technical reasons, we will need  to use two other sets of coefficients.
First define  $b'_{j,X,\bm{\a},\bm{\b}}(n)$ to be the $m$-fold  convolution of $h_X(n;\b_k) n^{-i\a_k} 1_{p|n \Rightarrow p\in (T_0,T_1]}$ when $j=1$ and  the $m$-fold  convolution  of $g_X(n;\b_k) n^{-i\a_k}1_{p|n \Rightarrow p\in (T_{j-1},T_j]}$ when $2\leq j \leq L$.
Finally, let $b''_{1,X,\bm{\a},\bm{\b}}(n)$ be the $m$-fold  convolution of $h_X(n;\b_k) 1_{p|n \Rightarrow p\in (T_0,T_1]}$ when $j= 1$ and the $m$-fold  convolution of $g_X(n;\b_k) 1_{p|n \Rightarrow p\in (T_{j-1},T_j]}$ when  $2\leq j \leq L$.
Unlike $b_{j,X,\bm{\a},\bm{\b}}$, the coefficients $b'_{j,X,\bm{\a},\bm{\b}}$ and $b''_{j,X,\bm{\a},\bm{\b}}$ are multiplicative, and they satisfy the bound $|b_{j,X,\bm{\a},\bm{\b}}(n)| , |b'_{j,X,\bm{\a},\bm{\b}}(n)| \leq b''_{j,X,\bm{\a},\bm{\b}}(n)$.
We will require the following information about these coefficients.

\begin{lemma}\label{lem:nCoeffPatrol}

For $1\leq j \leq L$ and $p \in (T_{j-1},T_j]$
\[
b_{j, X, \bm{\a}, \bm{\b}}(p) = a_X(p) \sum_{k = 1}^m \b_k p^{-i\a_k},
\]
and $b''_{j, X, \bm{\a}, \bm{\b}}(p) \leq \b_\ast$. If $r \geq 2$ 
\[
b''_{j, X, \bm{\a}, \bm{\b}}(p^r) \leq \frac{\b_\ast^r m^r}{r!}
\]
holds whenever $2\leq j \leq L$ or $p > \log T$, and otherwise
\[
b''_{1, X, \bm{\a}, \bm{\b}}(p^r) \leq  m \b_\ast^r r^{2m} e^{-r \log (r/m)/2m + 2r}.
\]

\end{lemma}
\begin{proof}
The first two assertions are immediate from the definition of the Dirichlet convolution.
To prove the upper bound when $r \geq 2$ and $j \neq 1$, first note that 
\begin{align*}
b_{j, X, \bm{\a}, \bm{\b}}(p^r) \leq \sum_{r_1 + \cdots + r_m = r} \prod_{l = 1}^m \frac{\b_l^{r_l}}{r_l!} \leq  \frac{\b_\ast^r}{r!} \sum_{r_1 + \cdots + r_m = r}  \binom{r}{r_1,\ldots ,r_m} =  \frac{\b_\ast^r m^r}{r!}.
\end{align*}
To handle the $j=1$ case, we can bound $h_X(p^r;\b)$ by 
\[
\sum_{t = 0}^{r/2} \frac{\b^{r-t}}{2^t t! (r-2t)!} \leq \b_\ast^r \sum_{t = 0}^{r/2} \frac{1}{2^t t! (r-2t)!}.
\]
In fact when $p >\log T$ we have the stronger bound $\b_r/r!$.
To bound the sum on the right hand side, note by Stirling's formula the maximum summand occurs near the solution to $(r - 2t)^2 = 2t$.
One more application of Stirling's formula shows that the maximum is $\leq e^{-r \log r / 2 + 2r}$, so this sum is bounded by $re^{-r \log r / 2 + 2 r}$.
It now follows that
\begin{align*}
|b_{1, X, \bm{\a}, \bm{\b}}(p^r)| \leq \b_\ast^r r^m \sum_{r_1 + \cdots + r_m = r} \prod_{l = 1}^m e^{-r_l \log r_l/2 + 2 r_l} \leq \b_\ast^r r^m e^{-r \log (r/m) /2m + 2 r} \binom{m+r - 1}{r},
\end{align*}
where we have used the fact that at least one $r_l$ must exceed $r/m$.
To conclude, notice that $\binom{m+r - 1}{r}$ is a polynomial of degree $m-1$ in $r$ with coefficients all bounded by 1, so it is at most $m r^{m-1}$.

\end{proof}
\noindent
We can now compute

\begin{prop}\label{prop:NMoment}
For $1\leq j  \leq L$ 
\begin{align*}
\int_{T}^{2T} \prod_{k = 1}^m &|\n_{j, X}(\tfrac{1}{2} + i (t + \a_k); \b_k)|^2 dt 
\\ &\leq (T + O(T^{1/2})) \prod_{p\in (T_{j-1},T_j]} \left(1 + \frac{|b_{j, X, \bm{\a}, \bm{\b}}(p)|^2}{p} + O_{\bm{\b}}\left(\frac{1}{p^2}\right)\right) + O_{\bm{\b}}(e^{- 50\b_\ast^2 K_j }).
\end{align*}
\end{prop}
\begin{proof}
By Lemma \ref{lem:MVDP} the mean value of interest equals 
\[
(T + O(T^{1/2}))  \sum_n \frac{|b_{j, X, \bm{\a}, \bm{\b}}(n)|^2}{n}.
\]
We will now show that we may replace $b$ with $b'$ at a negligible cost.
If $b_{j, X, \bm{\a}, \bm{\b}}(n) \neq b'_{j, X, \bm{\a}, \bm{\b}}(n)$ then it follows that $\Omega(n) \geq  100\b_\ast^2 K_j$, where $\Omega(n)$ is the number of prime factors of $n$ counting multiplicity.
Therefore when we replace $b$ with $b'$, we incur an error of at most
\[
e^{- 100\b_\ast^2 K_j} \sum_{p|n \Rightarrow p\in(T_{j-1},T_j]} \frac{b''_{j, X, \bm{\a}, \bm{\b}}(n)^2 e^{\Omega(n)}}{n}.
\]
Since the coefficients $b''$ are multiplicative, this is
\[
\ll e^{-100\b_\ast^2 K_j} \prod_{p\in (T_{j-1},T_j]} \left(1 + \frac{\b_\ast^2 e}{p} + O \left(\frac{\b_
\ast^4 m^4}{p^2}\right)\right) \ll_{\bm{\b}} e^{- 50\b_\ast^2K_j },
\]
where we have used Lemma \ref{lem:nCoeffPatrol} to bound the sum over prime powers.
Therefore the mean value of interest is 
\[
\leq (T + O(T^{1/2}))\sum_{p|n \Rightarrow p\in(T_{j-1},T_j]}  \sum_{r\geq 0} \frac{|b'_{j, X, \bm{\a}, \bm{\b}}(p^r)|^2}{p^r} + O_{\bm{\b}}(e^{- 50\b_\ast^2 K_j }).
\]
The claim now follows by Lemma \ref{lem:nCoeffPatrol} and multiplicativity.

\end{proof}
\noindent
We may finally deduce

\begin{prop}\label{prop:goodShiftedMoment}
Assuming the Riemann hypothesis
\begin{align*}
\int_{\G_{[m]}} \prod_{k\leq m} |\zeta(\tfrac{1}{2} + i (t + \a_k))|^{2 \b_k} dt \ll_{\bm{\b}} 
T (\log T)^{\b_1^2 + \cdots + \b_m^2} \prod_{1\leq j < k \leq m} |\zeta(1 + i(\a_j - \a_k) + 1/ \log T )|^{2\b_j \b_k}.
\end{align*}
\end{prop}

\begin{proof}
We have shown that the shifted moment on the good set is bounded by
\[
\ll T \prod_{j \leq L} \left(\prod_{p\in (T_{j-1},T_j]} \left(1 + \frac{|b_{j, X, \bm{\a}, \bm{\b}}(p)|^2}{p} + O_{\bm{\b}}\left(\frac{1}{p^2}\right)\right) + O_{\bm{\b}}(e^{-50\b_\ast^2K_j }) \right).
\]
To conclude, we first  note that 
\[
|b_{j, X, \bm{\a}, \bm{\b}}(p)|^2 \leq \sum_{j,k\leq m} \frac{\b_j \b_k} {p^{i(\a_j - \a_k)}} = \sum_{j\leq m} \b_j^2 + 2 \sum_{1\leq j < k\leq m} \b_j \b_k \cos((\a_j - \a_k) \log p) 
\]
and then use \ref{lem:cosPrimeSum} to bound the resulting products over primes.
\end{proof}

\section{Moments over bad shifts}\label{sec:BadShifts}
We now consider the integral over $\G_A$ where $A$ is a proper subset of $[m]$.
Without loss of generality we will write $A = [m]\setminus [a]$.
For each $t \in \G_A$, there is a function $F_t: [a] \rightarrow [L]$ such that $t + \a_j \in \B_{f(j)}$.
We will further partition $\G_A$ into the sets
\[
\B_{A,n} = \{t\in \G_A: \min_{j \in [a]} F_t(j) = n \}.
\]
First we handle the case of $n = 1$.

\begin{prop}\label{prop:B1BadMoment}
Assuming the Riemann hypothesis
\[
\int_{\B_{A,1}} \prod_{k = 1}^m |\zeta(\tfrac{1}{2}+ i(t + \a_k))|^{2\b_k} dt \ll_{A, \bm{\b}} T (\log T)^{-A}.
\]
\end{prop}
\begin{proof}
Because $\B_{A,1}$ is contained in the union of the $a$ translates $\B_1 - \a_j$ for $j \leq a$, the bound is a consequence of Lemma \ref{lem:measB1}, the Cauchy-Schwarz inequality, and Harper's \cite{Harper} bound for the moments of zeta (say).
\end{proof}

Now for fixed $n > 1$  we may use Lemma \ref{lem:logZetaUpperBound}  with $X = T_{n-1}$  to find
\begin{align*}
\int_{\B_{A,n}} \prod_{k = 1}^m &|\zeta(\tfrac{1}{2} + i (t + \a_k))|^{2 \b_k} dt \\
&\ll
\int_{\B_{A,n}} \prod_{k = 1}^m \exp\Bigg(2\b_k \Re \Bigg( \sum_{j < n}  \p_{j,T_{n-1}} (\tfrac{1}{2}  + i(t + \a_k)) + 2\b_k/ c_{n - 1} \Bigg) dt\\
&\ll  ~e^{2\b_\ast/ c_{n - 1}}  \int_{\B_{A,n}}\prod_{k = 1}^m \prod_{j < n} \exp\left(2\b_k \Re \p_{j,T_{n-1}} (\tfrac{1}{2}  + i(t + \a_k)) \right)dt \\
&\ll  ~e^{2\b_\ast/ c_{n - 1}} \max_{\substack{\ell \in [a] \\ s \in [L]}}
\int_T^{2T} \prod_{k = 1}^m \prod_{j < n} |\n_{j,T_{n-1}} (\tfrac{1}{2}+ i(t + \a_k) ; \b_k) |^2 \\
&\qquad  \qquad \qquad \qquad \qquad\times |\p_{n,T_s}(\tfrac{1}{2} + i (t + \a_\ell))/K_{n}|^{2 \lceil 1/ 10 c_{n}\rceil} dt
\end{align*}
Unlike the previous section, we have now also used the definition of the bad set $\B_{A,n}$.
By Lemma \ref{lem:Splitting}, all that remains is to control the moments of $\p_{n,T_s}$  on the half line.

\begin{prop}\label{prop:PMoment}
Uniformly for $\ell \in [m]$ and $s \in [L]$
\[
\int_{T/2}^{5T/2} |\p_{n,T_s}(\tfrac{1}{2} + i (t + \a_\ell))/K_{n}|^{2 \lceil 1/ 10 c_{n}\rceil}  dt \ll T e^{-\log(1/c_{n})/20 c_{n}}.
\]
\end{prop}
\begin{proof}
Trivially bounding $p^{-i\a_\ell}$ and $a_{X_s}(p)$ by 1, Lemma \ref{lem:MVDPPrimes} gives a bound of 
\[
T K_{n}^{-2r} r! \left(\sum_{p\in(T_{n-1},T_{n}]} \frac{1}{p} \right)^r
\]
where $r  = \lceil 1/ 10 c_{n}  \rceil$.
The sum in parentheses is asymptotic to $\log(c_{n}/c_{n - 1}) = e$, so is at most $2e$ for large $T$, say.
The conclusion follows by recalling $K_{n} = c_{n}^{-3/4}$ and applying Stirling's approximation. 
\end{proof}

We now have all the necessary tools to bound the shifted moment (\ref{eqn:shiftedMoments}) over the bad sets.

\begin{prop}
Assuming the Riemann hypothesis
\[
\int_{[T,2T]\setminus \G_{[m]}} \prod_{k = 1}^m |\zeta(\tfrac{1}{2} + i (t + \a_k))|^{2 \b_k} dt \ll_{A,\bm{\b}}  T (\log T)^{\b_1^2 + \cdots + \b_m^2} \prod_{1\leq j < k \leq m} |\zeta(1 + i(\a_j - \a_k) + 1/ \log T )|^{2\b_j \b_k}.
\]
\end{prop}

\begin{proof}

Applying Lemma \ref{lem:Splitting} along with Propositions \ref{prop:NMoment}, \ref{prop:B1BadMoment}, and \ref{prop:PMoment}, we may bound the relevant integral by
\begin{align*}
\ll_{A,\bm{\b}} T\sum_{2 \leq n\leq L}  \exp&\left(\frac{2 \b_\ast}{c_{n-1}} - \frac{\log(1/c_n)}{20c_n}\right) \prod_{ p \leq T_{n-1}} \left(1 + \sum_{1\leq j,k \leq m} \frac{\b_j \b_k}{p^{i(\a_j-\a_k)}} + O_{\bm{\b}}\left(\frac{1}{p^2}\right)\right) + T(\log T)^{-A} \\
\ll_{A,\bm{\b}} T &\sum_{n\leq L}   \exp\left(e^{-n}(\log_2 T)^2 (2\b_\ast e + \tfrac{1}{20}n - \tfrac{1}{10} \log_3 T) \right)  \\
&\qquad \qquad \times \prod_{ p \leq T_{n-1}} \left(1 + \sum_{1\leq j,k \leq m} \frac{\b_j \b_k}{p^{i(\a_j-\a_k)}} \right)   + T(\log T)^{-A},
\end{align*}
where we have applied a union bound over all bad subsets $A$ of $[m]$.
Because the shifts satisfy $|\a_j - \a_k| \leq T$, the $T(\log T)^{-A}$ term is negligible by the  estimate 
\[
|\zeta(1 + 1/ \log T + i t)|  \gg \zeta(2 + 2/\log T)/\zeta(1 + 1/\log T) \gg  1/\log T
\]
for $|t|\leq 2T$. To simplify remaining term, note that because $T_L \leq T^{e^{-1000\b_\ast}}$ it follows that $L \leq 2\log_3 T - 1000\b_\ast$.
Therefore the latter term is
\begin{align*}
\ll_{A,\bm{\b}} T 
\sum_{n\leq L}   \exp\left(-4\b_\ast e^{-n} (\log_2 T)^2 \right)\prod_{ p \leq T_{n-1}} \left(1 + \sum_{1\leq j,k \leq m} \frac{\b_j \b_k}{p^{i(\a_j-\a_k)}} \right)  \\
\ll_{A,\bm{\b}} T (\log T)^{\b_1^2 + \cdots + \b_m^2} \prod_{1\leq j < k \leq m} |\zeta(1 + i(\a_j - \a_k) + 1/ \log T )|^{2\b_j \b_k} \\
\times \sum_{n \leq L} \exp\left(-4\b_\ast e^{-n} (\log_2 T)^2 \right) \prod_{ p \in (T_{n-1},T_L]} \left(1 + \sum_{1\leq j,k \leq m} \frac{\b_j \b_k}{p^{i(\a_j-\a_k)}} \right)^{-1} \\
\ll_{A,\bm{\b}} T (\log T)^{\b_1^2 + \cdots + \b_m^2} \prod_{1\leq j < k \leq m} |\zeta(1 + i(\a_j - \a_k) + 1/ \log T )|^{2\b_j \b_k} \\
\times \sum_{n \leq L} \exp\left(-4\b_\ast e^{-n} (\log_2 T)^2 + \b_\ast^2 (L-n)\right).
\end{align*}
Note we used Merten's estimate when passing to final line. 
By summing in reverse, one readily sees the sum over $n$ is convergent, and the claim now follows.

\end{proof}

\noindent
In view of (\ref{eqn:GoodBadPartition}), this completes the proof of Theorem \ref{thm:main}.

\end{document}